\documentclass[11pt,leqno]{article}
\usepackage{ graphicx, amsfonts, amsthm, amsxtra, verbatim, multicol}
\usepackage[mathscr]{euscript}
\begin{document}

\begin{center}
{\LARGE\bf On elliptic modular foliations}
\\
\vspace{.25in} {\large {\sc Hossein Movasati}}\footnote{ Supported by Japan Society for the 
Promotion
of  Sciences.  
\\
Math. classification: 57R30 - 32G34
} \\
Instituto de Matem\'atica Pura e Aplicada, IMPA, \\
Estrada Dona Castorina, 110,\\
22460-320, Rio de Janeiro, RJ, Brazil, \\
E-mail: {\tt hossein@impa.br} \\
{\tt www.impa.br/$\sim$hossein }
\end{center}

\newtheorem{theo}{Theorem}
\newtheorem{exam}{Example}
\newtheorem{coro}{Corollary}
\newtheorem{defi}{Definition}
\newtheorem{prob}{Problem}
\newtheorem{lemm}{Lemma}
\newtheorem{prop}{Proposition}
\newtheorem{rem}{Remark}
\newtheorem{conj}{Conjecture}
\newcommand\diff[1]{\frac{d #1}{dz}} 
\def\End{{\rm End}}              
\def\hol{{\rm Hol}}
\def\sing{{\rm Sing}}            
\def\spec{{\rm Spec}}            
\def\cha{{\rm char}}             
\def\Gal{{\rm Gal}}              
\def\jacob{{\rm jacob}}          
\newcommand\Pn[1]{\mathbb{P}^{#1}}   
\def\Z{\mathbb{Z}}                   
\def\OO{{\cal O}}                       

\def\Q{\mathbb{Q}}                   
\def\C{\mathbb{C}}                   
\def\as{\mathbb{U}}                  
\def\ring{{R}}                             
\def\R{\mathbb{R}}                   
\def\N{\mathbb{N}}                   
\def\A{\mathbb{C}}                   
\def\D{\mathbb{D}}                   
\def\uhp{{\mathbb H}}                
\newcommand\ep[1]{e^{\frac{2\pi i}{#1}}}
\newcommand\HH[2]{H^{#2}(#1)}        
\def\Mat{{\rm Mat}}              
\newcommand{\mat}[4]{
     \begin{pmatrix}
            #1 & #2 \\
            #3 & #4
       \end{pmatrix}
    }                                
\newcommand{\matt}[2]{
     \begin{pmatrix}                 
            #1   \\
            #2
       \end{pmatrix}
    }
\def\ker{{\rm ker}}              
\def\cl{{\rm cl}}                
\def\dR{{\rm dR}}                

\def\hc{{\mathsf H}}                 
\def\Hb{{\cal H}}                    
\def\GL{{\rm GL}}                
\def\pedo{{\cal P}}                  
\def\PP{\tilde{\cal P}}              
\def\cm {{\cal C}}                   
\def\K{{\mathbb K}}                  
\def\k{{\mathsf k}}                  
\def\F{{\cal F}}                     
\def\M{{\cal M}}
\def\RR{{\cal R}}
\newcommand\Hi[1]{\mathbb{P}^{#1}_\infty}
\def\pt{\mathbb{C}[t]}               
\def\W{{\cal W}}                     
\def\Af{{\cal A}}                    
\def\gr{{\rm Gr}}                
\def\Im{{\rm Im}}                
\newcommand\SL[2]{{\rm SL}(#1, #2)}    
\newcommand\PSL[2]{{\rm PSL}(#1, #2)}  
\def\Res{{\rm Res}}              

\def\L{{\cal L}}                     
\def\Aut{{\rm Aut}}              
\def\any{R}                          
\newcommand\ovl[1]{\overline{#1}}    

\def\per{{\sf  pm}}  
\def\T{{\cal T }}                    
\def\tr{{\sf tr}}                 
\newcommand\mf[2]{{M}^{#1}_{#2}}     
\newcommand\bn[2]{\binom{#1}{#2}}    
\def\ja{{\rm j}}                 
\def\Sc{\mathsf{S}}                  
\newcommand\es[1]{g_{#1}}            
\newcommand\V{{\mathsf V}}           
\newcommand\Ss{{\cal O}}             
\def\rank{{\rm rank}}                
\def\diag{{\rm diag}}

\def\Ra{\mathrm{Ra}}
\def\nf{g_2}                         
\begin{abstract}
In this article we consider the three parameter  family of elliptic curves 
$E_t: y^2-4(x-t_1)^3+t_2(x-t_1)+t_3=0,\ t\in\C^3$ 
and study the modular holomorphic foliation $\F_{\omega}$ in $\C^3$ whose leaves 
 are constant locus of the integration of a 1-form $\omega$ over topological cycles of $E_t$.
 Using the Gauss-Manin connection of the family $E_t$, we show  that $\F_{\omega}$ is
 an algebraic foliation. In the case $\omega=\frac{xdx}{y}$, 
 we prove  that a transcendent leaf of $\F_{\omega}$ contains at most
 one point with algebraic coordinates and the leaves of $\F_{\omega}$ 
 corresponding to the zeros of integrals, never 
cross such a point. 
Using the generalized period map associated to the family $E_t$,
  we find  a uniformization of  $\F_{\omega}$ in $T$, 
  where $T\subset \C^3$ is the locus
  of parameters $t$ for which $E_t$ is smooth. 
  We find also a real first integral  of $\F_\omega$ restricted to $T$ 
  and show that
 $\F_{\omega}$  is given by  the  Ramanujan relations between the Eisenstein series. 
 
\end{abstract}
\section{Introduction}
A classical way to study an object in algebraic geometry, is to put it inside 
a family and then try to understand its behavior as a member of  the family.
In other words, one looks the object inside a certain moduli space.
The abelian integrals which appear in the deformation of holomorphic foliations
with a first integral in a complex manifold of dimension two 
(see \cite{ily, ga02, mov0, mov3}), can be studied in this way provided that we consider, apart 
from the parameter of the first integral, some other parameters. 
The first natural object to look is the constant locus of integrals.
This yields to holomorphic foliations in the parameter space, which we call
modular foliations.  The defining equations of such holomorphic
foliations can be calculated using the Gauss-Manin connection and  it turns
out that they are always defined over $\Q$, i.e. the ingredient of the defining equations
are polynomials in the parameters and with coefficients in $\Q$.
Modular foliations, apart from topological and dynamical properties, enjoy
certain arithmetical properties. They are an important link between the transcendental problems in number theory and  their counterparts in holomorphic foliations/differential equations.
They are classified as transversely homogeneous  foliations (see \cite{god}) 
and recently some  authors have studied examples  of  such foliations (see \cite{sc97, casc99, li04, lope} and
the references there). 
In this article I want to report on a class of such foliations associated to a three parameter family
of elliptic curves. For simplicity, we explain the results of this article for one
of such foliations which is important from historical point of view and its transverse group structure is 
$\SL 2\Z$.


After calculating the Gauss-Manin connection of the following family of 
elliptic curves
\begin{equation}
\label{khodaya}
E_t: y^2-4(x-t_1)^3+t_2(x-t_1)+t_3=0,\ t\in\C^3
\end{equation}
and  considering its relation with the inverse of the period map, we get the
following ordinary differential equation:
\begin{equation}
\label{raman}
\Ra: \left \{ \begin{array}{l}
\dot t_1=t_1^2-\frac{1}{12}t_2 \\
\dot t_2=4t_1t_2-6t_3 \\
\dot t_3=6t_1t_3-\frac{1}{3}t_2^2
\end{array} \right.
\end{equation}
which is called the Ramanujan relations,
because he has observed that the Eisenstein series form a 
solution of (\ref{raman}) (one gets  the classical relations by changing the coordinates
$(t_1,t_2,t_3)\mapsto (\frac{1}{12}t_1,\frac{1}{12}t_2,
\frac{2}{3(12)^2}t_3)$, see \cite{nes01}, p. 4). We denote by
$\F(\Ra)$ the singular holomorphic foliation induced by (\ref{raman}) in
$\C^3$. Its singularities
$$
\sing(\Ra):=\{(t_1,12t_1^2,8t_1^3)\mid t_1\in \C \}
$$
form  a one-dimensional curve in $\C^3$.  The discriminant of the family (\ref{khodaya}) is
given by $\Delta=27t_3^2-t_2^3$. For $t\in T:=\C^3\backslash \{\Delta=0\}$, $E_t$ is an smooth
elliptic curve and so  we can take  a basis of the $\Z$-module 
$H_1(E_t,\Z)$, namely $(\delta_1,\delta_2)=(\delta_{1,t},\delta_{2,t})$,  such that the intersection matrix in 
this basis is
$\mat{0}{1}{-1}{0}$. Let $\omega_i,\ i=1,2$ be two meromorphic differential 1-forms
in $\C^2$ such that the restriction of $\omega_i$ to $E_t,\ t\in T$ is of the second type, i.e 
it may have poles but no residues around the poles. For instance, take $\omega_1=\frac{dx}{y},\ \omega_2=
\frac{xdx}{y}$. Define
$$
B_{\omega_i }(t):=\frac{1}{2\pi}\Im \left (\int_{\delta_{1}}\omega_i\overline{\int_{\delta_{2}}\omega_i}\right ),\ i=1,2,
$$
$$
B_{\omega_1,\omega_2}(t):=
\frac{1}{2\pi}\left (\int_{\delta_{1}}\omega_1\overline{\int_{\delta_{2}}\omega_2}-\int_{\delta_{1}}\omega_2\overline{\int_{\delta_{2}}\omega_1}\right ).
$$
It is easy to show that the above functions do not depend on the choice of
$\delta_1,\delta_2$ (see  the definition of the period map in 
\S\ref{ramanfoli}) and hence they define  analytic functions
on $T$.

We define
 $$
K:={\Big \{}t\in T\mid \int_{\delta}\frac{xdx}{y}=0,\ \hbox{ for some }
0\not =\delta\in H_1(E_t,\Z){\Big \}}
$$
and
$$
M_r:=\{t\in T \mid B_{\frac{xdx}{y}}(t)=r\},\ M_{<r}:=\cup_{s<r}M_s,\  r\in\R. 
$$
Using the Legendre relation $\int_{\delta_{1}}\frac{dx}{y}\int_{\delta_{2}}\frac{xdx}{y}-
\int_{\delta_{1}}\frac{xdx}{y}\int_{\delta_{2}}\frac{dx}{y}=2\pi i$ one can show that $|B_{\frac{dx}{y}, \frac{xdx}{y}}|$ restricted
to $M_0$ is identically $1$. We also define 
$$
N_w:=\{t\in M_0\mid B_{\frac{dx}{y}, \frac{xdx}{y}}(t)=w \}, \ |w|=1, \ w\in\C.
$$
For $t\in \C^3\backslash\sing(\F(\Ra))$ we denote by $L_t$ the leaf
of $\F(\Ra)$ through $t$. Let $\uhp:=\{z\in\C\mid \Im(z)>0\}$ be the Poincar\'e
upper half plane and $\D:=\{z\in \C\mid |z|<1\}$ be the unit disk.
\begin{theo}
\label{foli}
The following is true:
\begin{enumerate}
\item
The leaves of $\F(\Ra)$ in a neighborhood of $t\in T$ are given
by the level surfaces of $(\int_{\delta_{1}}\frac{xdx}{y},\int_{\delta_{2}}\frac{xdx}{y}):(T,t)\rightarrow 
\C^2$. In particular, the function $B_{\frac{xdx}{y}}$ is a real first integral of
$\F(\Ra)$ and  for $|w|=1$, $N_w$'s are $\F(\Ra)$-invariant. 
\item
For $t\in X:=(M_0\backslash K)\cup M_{<0}$ the leaf $L_t$ is biholomorphic to $\D$ and
for $t\in T\backslash X$ the leaf $L_t$ is biholomorphic to $\D\backslash \{0\}$.
\item
The set $K$ is $\F(\Ra)$-invariant and it is a dense subset of $M_0$. For all $t\in K$ there is a holomorphic map $\D\rightarrow \C^3$, transverse to 
$\sing(\Ra)$ at some point $p$, which  is a biholomorphy between $\D\backslash\{0\}$ and $L_t$.
\item
For all $t\in T$ the leaf $L_t$ has an accumulation point at $T$ if
and only if $t\in M_0$.
\item
The discriminant variety  $\{\Delta=0\}$ is 
$\F(\Ra)$-invariant and all the leaves in $\{\Delta=0\}$ are algebraic.
\end{enumerate}
\end{theo}
In \S\ref{emf} we have defined an elliptic modular foliation associated to a differential
form $\omega$ in $\C^2$ such that $\omega$ restricted to the fibers of (\ref{khodaya}) is holomorphic. It is based on the first  statement in Theorem \ref{foli}, part 1. Such foliations have real first integrals and leave the discriminant variety invariant. 

The proof of the above theorem is based on the fact that the foliation
$\F(\Ra)$ restricted to $T$ is  uniformized by the inverse of the 
period map (see for instance \cite{li00} for similar topics).
Despite the fact that this theorem does not completely describe the dynamics of $\F(\Ra)$, it
shows that a modular foliation is not a strange foliation from dynamical/topological point of view.  
However, such foliations arise some new questions and problems related to holomorphic foliations.
For a given algebraic holomorphic  foliation $\F$ in $\C^3$ defined over $\bar \Q$, the field of algebraic numbers,
a transcendent leaf $L$ of $\F$ how frequently crosses points with algebraic coordinates? 
The set $L\cap\bar \Q^3$ can be empty or
a one element set. For $\F(\Ra)$ these are the only possibilities.
\begin{theo}
\label{21feb06}
The following is true:
\begin{enumerate}
\item
For any point $t\in \C^3\backslash \{\Delta=0\}$, the set $\bar\Q^3 \cap L_t$ is empty or has only one element. In other words, every transcendent 
leaf contains at most one point with algebraic coordinates.
\item
$K\cap \bar \Q^3=\emptyset$, i.e for all $p\in K$ at least one of the coordinates of
$p$ is transcendent number.
\end{enumerate}
\end{theo} 

The main idea behind the proof of the above theorem is the first 
part of Theorem \ref{foli} and consequences of 
the abelian subvariety theorem on periods of elliptic curves (see \cite{wowu} and the
references there).  We will also give an alternative proof for
the second part of the above theorem, using a result on transcendence of
the values of the Eisenstein series.

I have made a good use of {\sc Singular} for doing the calculations in this article. 
 The text is written in such a way that the reader can carry out all calculations using
any software in commutative algebra. An exception to this is 
the calculation of the Gauss-Manin connection in \S \ref{ramanfoli}, for which one can
use a combination of  hand and computer calculations or one must know the general algorithms introduced in
\cite{hos005}. 
The general definition of a modular foliation can be done
using connections on algebraic varieties. 
The forthcoming text \cite{hos2006} will discuss such foliations, specially
those related to the Gauss-Manin connection of fibrations.
In the article \cite{hos06} we have developed the notion of a
differential modular form in which we have essentially used the same 
techniques of this article.

In the classical theory of elliptic integrals, the parameter $t_1$ in (\ref{khodaya}) is equal to zero
and one considers the versal deformation of the singularity $y^2-4x^3=0$. 
In this article we have generalized the classical Weierstrass Theorem and proved  that for the inverse
of the generalized period map, $t_i$ appears as the Eisenstein series of weight $2i$.
The novelty is the appearance of $t_1$ as the Eisenstein series of weight $2$.

The paper is organized as follows: In \S\ref{ramanfoli}
we define the period map, calculate its derivative and the Gauss-Manin connection
associated to the family (\ref{khodaya}). In \S\ref{alggr} we introduce the action of an algebraic group 
on $\C^3$ and its relation with the period map. We prove that the period map  
is a biholomorphism and using its inverse, we 
obtain the differential equation (\ref{raman}). In \S\ref{uniformi} we describe the uniformization
of $\F(\Ra)\mid_T$. In \S\ref{proofof} we prove
Theorem \ref{foli}. In \S \ref{emf} we introduce the general notion of an elliptic modular foliation associated to the family (\ref{khodaya}). 
\S \ref{ast} is devoted to a theorem on periods of abelian varieties defined
over $\bar\Q$ and its corollaries on the periods of elliptic curves. In \S \ref{proof2} we prove
Theorem \ref{21feb06}. In \S \ref{sarakar} we study another family of elliptic curves and corresponding
modular foliations. Finally in \S \ref{21febr} we discuss some problems related to limit cycles
arising from deformations of the family (\ref{khodaya}) inside holomorphic foliations.

{\it Acknowledgment:}  During the preparation of this text, I visited IMPA at Rio de Janeiro.
Here I would like to thank the institute and the  participants of  the complex dynamics seminar.
In particular, I would like to thank C. Camacho for his comments and 
J. V. Pereira, who pointed out that the elliptic modular foliations
are transversely homogeneous foliations.
\section{Period map and its derivation}
\label{ramanfoli}
For some technical reasons, which will be clear later, it is convenient to introduce
a new parameter $t_0$ and work with the family:
\begin{equation}
\label{khodaya1}
E_t: y^2-4t_0(x-t_1)^3+t_2(x-t_1)+t_3,\ t=(t_0,t_1,t_2,t_3)\in\C^4.
\end{equation}
Its discriminant is $\Delta:=t_0(27t_0t_3^2-t_2^3)$. We will use the notations 
in the Introduction for this family.

Let 
$$
\pedo:=\{x=\mat {x_1}{x_2}{x_3}{x_4}\in \GL(2,\C)\mid \Im(x_1\ovl{x_3})>0\}.
$$
It is well-know that the entries of $(\omega_1,\omega_2):=(\frac{dx}{y},\frac{xdx}{y})$
restricted to each regular elliptic curve $E_t$ form a basis 
of $H^1_{\dR}(E_t)$.
The associated 
period map is given by:
$$
\per: T\rightarrow \SL 2\Z\backslash \pedo,\ t\mapsto
\left [\frac{1}{\sqrt{2\pi i}}\mat
{\int_{\delta_1}\omega_1}
{\int_{\delta_1}\omega_2}
{\int_{\delta_2}\omega_1}
{\int_{\delta_2}\omega_2} \right ].
$$
It is well-defined and holomorphic. Here $\sqrt{i}=e^{\frac{2\pi
i}{4}}$ and $(\delta_1,\delta_2)$ is a basis of the $\Z$-module 
$H_1(E_t,\Z)$
such that the intersection matrix in this basis is
$\mat{0}{1}{-1}{0}$. Note that $\delta_i=\delta_{i,t},\ i=1,2$ is a continuous family of cycles
depending on $t$. 
Different choices of $\delta_1,\delta_2$ will
lead to the action of $\SL 2\Z$ on $\pedo$ from the left. If there is
no risk of confusion, we will also use $\per$ for the map from $T$ 
to $\pedo$. 
\begin{rem}\rm
A classical way for choosing the cycles $\delta_1,\delta_2$ is given by
the Picard-Lefschetz theory (see for instance \cite{mov0} and the references there). 
For the fixed parameters $t_0\not =0,\ t_1$ and $t_2\not=0$, define 
$f:\C^2\rightarrow \C$ as 
$$
f(x,y)=-y^2+4t_0(x-t_1)^3-t_2(x-t_1).
$$ 
The function $f$
has two critical values given by $\tilde t_3, \ \check t_3=\pm \sqrt{\frac{t_2^3}{27t_0}}$. 
In a regular fiber $E_t$
of $f$ one can take two cycles $\delta_1$ and $\delta_2$ such that $\langle \delta_1,\delta_2\rangle=1$
and $\delta_1$ (resp. $\delta_2$) vanishes along a straight line connecting $t_3$ to $\tilde t_3$ (resp. 
$\check t_3$). The corresponding anti-clockwise monodromy around the critical value 
$\tilde t_3$ (resp $\check t_3$)  can be computed using the Picard-Lefschetz formula:
$$
\delta_1\mapsto \delta_1,\ \delta_2\mapsto \delta_2+\delta_1 \ 
(\hbox{ resp. } \delta_1\mapsto \delta_1-\delta_2,\ \delta_2\mapsto \delta_2).
$$ 
It is not hard to see that the canonical map $\pi_1(\C\backslash \{\tilde t_3,\check t_3\},t)\rightarrow 
\pi_1(T,t)$ induced by inclusion is an isomorphism of groups and so: 
$$
\pi_1(T,t)\cong\langle A_1,A_2\rangle=\SL 2\Z ,\ \hbox{ where } A_1:=\mat{1}{0}{1}{1},\ A_2:=\mat{1}{-1}{0}{1}.
$$
Note that  if we define $g_1:=A_2^{-1}A_1^{-1}A_2^{-1}=\mat{0}{1}{-1}{0},\ g_2:=A_1^{-1}A_2^{-1}=\mat{1}{1}{-1}{0}$ 
then we have $\SL 2\Z=\langle g_1,g_2\mid g_1^2=g_2^3=-I\rangle$, where $I$ is the identity $2\times 2$ matrix.
\end{rem}
\begin{prop}
\label{18.1.06}
Consider $\per$ as a holomorphic matrix valued function in $T$. We have
\begin{equation}
\label{4mar}
 d\per(t)=\per(t) \cdot A ^\tr,\ t\in T,
\end{equation}
where $A=\frac{1}{\Delta}\sum_{i=1}^4 A_idt_i$ and

\begin{equation}
\label{rosa} A_0 =\mat {\frac{3}{2}t_0t_1t_2t_3-9t_0t_3^2+\frac{1}{4}t_2^3}
{-\frac{3}{2}t_0t_2t_3}
{\frac{3}{2}t_0t_1^2t_2t_3+9t_0t_1t_3^2-\frac{1}{2}t_1t_2^3+\frac{1}{8}t_2^2t_3}
{-\frac{3}{2}t_0t_1t_2t_3-18t_0t_3^2+\frac{3}{4}t_2^3} 
\end{equation}
$$
A_1=\mat 0 0 {27t_0^2t_3^2-t_0t_2^3} 0
$$
$$
A_2 =\mat {-\frac{9}{2}t_0^2t_1t_3+\frac{1}{4}t_0t_2^2} {\frac{9}{2}t_0^2t_3}
{-\frac{9}{2}t_0^2t_1^2t_3+\frac{1}{2}t_0t_1t_2^2-\frac{3}{8}t_0t_2t_3}
{\frac{9}{2}t_0^2t_1t_3-\frac{1}{4}t_0t_2^2}
$$
$$
A_3 =\mat {3t_0^2t_1t_2-\frac{9}{2}t_0^2t_3} {-3t_0^2t_2}
{3t_0^2t_1^2t_2-9t_0^2t_1t_3+\frac{1}{4}t_0t_2^2}
{-3t_0^2t_1t_2+\frac{9}{2}t_0^2t_3}
$$
\end{prop}
\begin{proof}
The proof is a mere calculation.
The calculation of the derivative of the period map for
the differential form $\frac{dx}{y}$ and the case  $t_1=0$ is classical and can be  
found in (\cite{sas} p. 304, \cite{sai01} ). 
For the convenience of the reader we explain only the first row of $A_3$.
For $p(x)=4t_0(x-t_1)^3-t_2(x-t_1)-t_3$ we have:
$$
\Delta=-p'\cdot a_1+p\cdot a_2,
$$ 
where 
$$
a_1=
-36t_0^3x^4+144t_0^3t_1x^3+(-216t_0^3t_1^2+15t_0^2t_2)x^2+(144t_0^3t_1^3-30t_0^2t_1t_2)x-36t_0^3t_1^4+15t_0^2t_1^2t_2-t_0t_2^2
$$
$$
a_2=
(-108t_0^3)x^3+(324t_0^3t_1)x^2+(-324t_0^3t_1^2+27t_0^2t_2)x+(108t_0^3t_1^3-27t_0^2t_1t_2-27t_0^2t_3)
$$ 
Now we consider $y$ as a function in $x$ and make the projection of $H_1(E_t,\Z)$ in the
$x$-plane. The derivation with respect to $t_3$ goes inside of the integral and
\begin{eqnarray*}
\frac{\partial}{\partial t_3}(\frac{dx}{y}) &= & \frac{1}{2}\frac{dx}{py} = \frac{1}{\Delta}\frac{(-p'a_1+pa_2)dx}{2py}= \frac{1}{\Delta}(\frac{1}{2}a_2-a_1')\frac{dx}{y} \\
&=& (3t_0^2t_1t_2-\frac{9}{2}t_0^2t_3)\frac{dx}{y}
-3t_0^2t_2\frac{xdx}{y}
 \hbox{ modulo relatively exact 1-forms}
\end{eqnarray*}
(see \cite{prso} p. 41 for a description of calculations modulo relatively exact 1-forms).
Note that in the third equality above we use $y^2=p(x)$ and the fact that modulo exact forms we have
$$
\frac{p'a_1dx}{2py}=\frac{a_1dp}{2py}=\frac{a_1dy}{p}=-a_1d(\frac{1}{y})=\frac{a_1'dx}{y}.
$$

 \end{proof}
Recall that a meromorphic differential form $\omega$ in $\C^2$ is relatively exact for the family (\ref{khodaya}) 
if its restriction 
to each elliptic curve $E_t,\ \Delta(t)\not=0$ is an exact form. This is equivalent to say that $\int_{\delta}\omega=0$ for
all $\delta\in H_1(E_t,\Z)$.

The matrix $A$ is in fact the Gauss-Manin connection of the family
$E_t$ with respect to the basis $\omega$. We consider (\ref{khodaya1}) as an elliptic
curve  $E$ defined over $\Q(t)=\Q(t_0,t_1,t_2,t_3)$. According to  Grothendieck \cite{gro},  
the de Rham cohomology $H_\dR^1(E)$ of $E$  is well-defined. Any element of $H_\dR^1(E)$ can be  represented 
by a meromorphic differential 1-form in $\C^2=\{(x,y)\}$ whose restriction to
a generic  elliptic curve $E_t$ is a differential form of the second type i.e. a meromorphic
differential form on $E_t$ with no residues around its poles. In the case we are considering, each element in 
$H_\dR^1(E)$ can be represented by a differential form with a unique pole at infinity and $H_\dR^1(E)$ is a $\Q(t)$-vector space with the  basis $\{[\frac{dx}{y}],[\frac{xdx}{y}]\}$. 
Roughly speaking, the Gauss-Manin
connection is a $\Q$-linear operator 
$\nabla:H_\dR(E)\rightarrow \Omega^1_{T}\otimes_{\Q(t)}H_\dR(E)$,
where $\Omega^1_{T}$ is the set of algebraic differential 1-forms defined over $\Q$ 
in $T$. 
It satisfies
the Leibniz rule $\nabla(p\eta)=dp\otimes \eta+p\nabla\eta,\ p\in\Q(t),\ \eta\in 
H^1_\dR(E)$ and
\begin{equation}
\label{15feb06}
d\int_{\delta_t}\eta=\int_{\delta_t}\nabla \eta,\  \eta\in 
H^1_\dR(E).
\end{equation}
We write $\nabla(\omega)=B\omega, \ \omega:=(\frac{dx}{y},\frac{xdx}{y})^\tr$, 
use (\ref{15feb06}) and conclude that 
$B=\frac{1}{\Delta}\sum_{i=0}^3A_i$. 
The Gauss-Manin connection is an integrable connection. For our example, this translates into:
$$
dB=B\wedge B  \hbox{ equivalently for }  
B=\mat{\omega_{11}}{\omega_{12}}{\omega_{21}}{\omega_{22}}
$$
$$
d\omega_{11}=\omega_{12}\wedge\omega_{21},\
d\omega_{12}=\omega_{12}\wedge\omega_{22}+\omega_{11}\wedge\omega_{12},\ 
d\omega_{22}=\omega_{21}\wedge\omega_{12},\
d\omega_{21}=\omega_{21}\wedge\omega_{11}+\omega_{22}\wedge\omega_{21}.\ 
$$  
For the procedures which calculate the Gauss-Manin connection see
\cite{hos005}.
\section{Action of an algebraic group}
\label{alggr}
The algebraic group
\begin{equation}
\label{alggroup}
G_0=\left \{\mat{k_1}{k_3}{0}{k_2}\mid \ k_3\in\C, k_1,k_2\in \C^*\right \}
\end{equation}
acts on $\pedo$ from the right by the usual multiplication
of matrices. It acts also in $\C^4$ as follows:
\begin{equation}
\label{action}
t\bullet g:=(t_0k_1^{-1}k_2^{-1},
 t_1k_1^{-1}k_2+k_3k_1^{-1},
t_2k_1^{-3}k_2, t_3k_1^{-4}k_2^2) $$ $$
 t=(t_0,t_1,t_2,t_3)\in\C^4,
g=\mat {k_1}{k_3}{0}{k_2}\in G_0.
\end{equation}
The relation between these two actions of $G_0$ is given by:
\begin{prop}
\label{cano} The period $\per$ is a biholomorphism and
\begin{equation}
\label{gavril}
\per(t\bullet g)=\per(t)\cdot g,\ t\in\A^4,\ g\in G_0.
\end{equation}
\end{prop}
\begin{proof}
We first prove (\ref{gavril}). Let
$$
\alpha: \A^2\rightarrow \A^2,\ (x,y)\mapsto
(k_2^{-1}k_1x-k_3k_2^{-1}, k_2^{-1}k_1^{2}y).
$$
Then
$$
k_2^{2}k_1^{-4}\alpha^{-1}(f)=y^2-4t_0k_2^{2}k_1^{-4} (
k_2^{-1}k_1x-k_3k_2^{-1}-t_1)^3+ t_2k_2^{2}k_1^{-4}(
k_2^{-1}k_1x-k_3k_2^{-1}-t_1)+t_3k_2^{2}k_1^{-4}
$$
$$
y^2-4t_0k_1^{-1}k_2^{-1}(x-(t_1k_2k_1^{-1}+k_3k_1^{-1}))^3+
t_2k_1^{-3}k_2(x-(t_1k_2k_1^{-1}+k_3k_1^{-1}))+t_3k_1^{-4}k_2^{2}.
$$
This implies that $\alpha$ induces an isomorphism of elliptic curves
$$
\alpha: E_{t\bullet g}\rightarrow E_t
$$
Now
$$
\alpha^{-1} \omega= \mat{k_1^{-1}}{0}{-k_3k_2^{-1}k_1^{-1}}{k_2^{-1}}\omega=\mat {k_1}{0}{k_3}{k_2}^{-1} \omega,
$$
where $\omega=(\frac{dx}{y},\frac{xdx}{y})^\tr$, and so
$$
\per(t)= \per(t\bullet g).g^{-1}
$$
which proves (\ref{gavril}).

Let $B$ be a $4\times 4$
matrix and the $i$-th row of $B$ constitutes of the first and
second rows of $A_i$. We use the explicit expressions for $A_i$'s in Proposition \ref{18.1.06} and we
derive the following equality:
$$
\det(B)=\frac{3}{4}t_0\Delta^3
$$
The matrix $B$ is the derivation of the period map seen as a local function from $\C^4$ to $\C^4$. This 
shows that $\per$ is regular at each point $t\in
T$ and hence it is locally a biholomorphism.
The period map $\per$ induces a local biholomorphic map
$\bar \per: T/G_0\rightarrow \SL 2\Z\backslash \uhp\cong \C$.
One can compactify $\SL 2\Z\backslash \uhp$ by adding the cusp $\SL
2\Z/\Q=\{c\}$ (see \cite{la95}) and the map $ \bar \per $ is
continuous at $v$ and sends $v$ to $c$, where $v$ is the point induced by
$t_027t_3^2-t_2^3=0$ in $\A^4/G_0$.
 Using Picard's Great Theorem we conclude that $ \bar \per$ is a
biholomorphism and so $\per$ is a biholomorphism.
\end{proof}


We denote by
$$
F=(F_0,F_1,F_2,F_3):\pedo\stackrel{\alpha}{\rightarrow} \SL 2\Z\backslash \pedo \rightarrow T
$$
the map obtained by the composition of the canonical map $\alpha$ and the  inverse of the 
period map. 
Taking $F$ of (\ref{gavril}) we have
$$
F_0(xg)=F_0(x)k_1^{-1}k_2^{-1}, 
$$
\begin{equation}
\label{2apr05}
F_1(xg)=F_1(x)k_1^{-1}k_2+k_3k_1^{-1},\
\end{equation}
$$
F_2(xg)=F_2(x)k_1^{-3}k_2,\  F_3(xg)=F_3(x)k_1^{-4}k_2^2,\ \forall 
x\in\L,\
g\in G_0.
$$
By Legendre's  Theorem $\det(x)$ is equal to one on $\per 
(1\times 0\times \A\times \A)$ and so the same is true for
$F_0\det(x)$. But the last function is invariant under the action of
$G_0$ and so it is the constant function $1$.  This  means that
$F_0(x)=\det(x)^{-1}$. 

We consider $\per$ as a map sending the vector $(t_0,t_1,t_2,t_3)$ to
$(x_1,x_2,x_3,x_4)$. Its derivative at $t$ is a $4\times 4$ matrix whose
$i$-th column constitutes of the first and second row of
$\frac{1}{\Delta}xA_i^\tr$. We  use (\ref{rosa})
to derive the equality {\tiny
$$
(dF)_x=(d\per)_t^{-1}=
$$
$$
\det(x)^{-1}\left ( \begin{array}{cccc} -F_0x_4 &F_0x_3
&F_0x_2 &-F_0x_1
\\ \frac{1}{12F_0}(12F_0F_1^2x_3-12F_0F_1x_4-F_2x_3)
&-F_1x_3+x_4 &\frac{1}{12F_0}(-12F_0F_1^2x_1+12F_0F_1x_2+F_2x_1)
&F_1x_1-x_2
\\4F_1F_2x_3-3F_2x_4-6F_3x_3
&-F_2x_3 &-4F_1F_2x_1+3F_2x_2+6F_3x_1 &F_2x_1
\\\frac{1}{3F_0}(18F_0F_1F_3x_3-12F_0F_3x_4-F_2^2x_3)
&-2F_3x_3 &\frac{1}{3F_0}(-18F_0F_1F_3x_1+12F_0F_3x_2+F_2^2x_1)
&2F_3x_1
\end{array} \right ).
$$
} 
Define $g_i(z):=F_i \mat{z}{-1}{1}{0}, \ z\in\uhp$.
The equalities of the first column of the above
matrix imply that $(g_1,g_2,g_3): \uhp\rightarrow T$ satisfies the ordinary
differential equation (\ref{raman}). The equalities (\ref{2apr05}) imply that
$g_i$'s satisfy
\begin{equation}
\label{14feb06}
(cz+d)^{-2i}g_i(Az)=g_i(z),\  i=2,3,\
\end{equation}
\begin{equation}
\label{14feb}
 (cz+d)^{-2}g_1(Az)=g_1(z)+c(cz+d)^{-1}, \
\mat{a}{b}{c}{d}\in \SL 2\Z.
\end{equation}
In fact $g_i$'s, 
up to some constants, are the Eisenstein series. More precisely,
\begin{prop}
\label{16feb06}
We have
\begin{equation}
\label{eisenstein}
\es{k}(z)=a_k{\Big (}1+(-1)^k\frac{4k}{B_k}\sum_{n\geq
1}\sigma_{2k-1}(n)e^{2\pi i z n}{\Big )},\ \  k=1,2,3, \ z\in\uhp,
\end{equation}
where $B_k$ is the $k$-th Bernoulli number ($B_1=\frac{1}{6},\
B_2=\frac{1}{30},\ B_3=\frac{1}{42},\ \ldots$), $\sigma_i(n):=
\sum_{d\mid n}d^i$,
\begin{equation}
\label{pinfty}
p_\infty:=(a_1,a_2,a_3)=(\frac{2\pi i}{12},12(\frac{2\pi i}{12})^2 ,
8(\frac{2\pi i}{12})^3)
\end{equation}
\end{prop}
\begin{proof}
The statement for $g_2$ and $g_3$ follows from the Weierstrass 
uniformization Theorem (see for instance \cite{sai01}).  
Note that in our definition of the 
period map the factor $\frac{1}{\sqrt{2\pi i}}$ appears.
The functions $g_k, \ k=1,2,3$  have finite growth at infinity, i.e. 
$\lim_{\Im(z)\to +\infty}g_k(z)=a_k<\infty$. For $g_1$ this follows from
the Ramanujan relations (\ref{raman}) and the equality 
$\frac{d}{dz}=2\pi i q \frac{d}{dq}$, where $q=e^{2\pi i z}$.
The set $M$ of holomorphic functions on $\uhp$ which have finite growth at
infinity and satisfy (\ref{14feb}) contains only one element. The reason is as follows: 
The difference of
any two elements of $M$ has finite growth at infinity and satisfy (\ref{14feb06}) with $i=1$.
Such a holomorphic function is a modular form of weight $2$ which does not exist (see \cite{la95}). 
Now the function $g_1$ and its corresponding series in (\ref{eisenstein}) have finite
growth at infinity and satisfy (\ref{14feb}) (see \cite{apo90} p. 69). 
Therefore, they must be equal.
\end{proof}

\section{Uniformization of $\F(\Ra)$}
\label{uniformi}
From this section on, we set $t_0=1$ and  work again with the family (\ref{khodaya}). 
We use the same
notations for $\per,\pedo,G_0,T, \Delta$ and so on. For instance,
redefine
$$
\pedo:=\{x=\mat {x_1}{x_2}{x_3}{x_4}\in \GL(2,\C)\mid \Im(x_1\ovl{x_3})>0,\
\det(x)=1 \}
$$
and
$$
G_0=\{\mat{k}{k'}{0}{k^{-1}}\mid \ k'\in\C, k\in \C^*\}.
$$
The action of $G_0$ on $\A^3$ is given by
$$
t\bullet g:=(
 t_1k^{-2}+k'k^{-1},
t_2k^{-4}, t_3k^{-6}), t=(t_1,t_2,t_3)\in\A^3,
g=\mat {k}{k'}{0}{k^{-1}}\in G_0.
$$
We also define
$$
g=(g_1,g_2,g_3): \uhp\rightarrow T\subset \A^3,
$$
$$
\Ra:=
(t_1^2-\frac{1}{12}t_2)\frac{\partial }{\partial t_1}+
( 4t_1t_2-6t_3)\frac{\partial }{\partial t_2}+
(6t_1t_3 -\frac{1}{3}t_2^2)\frac{\partial }{\partial t_3},
$$
$$
\eta_{1}:=(t_1^2-\frac{1}{12}t_2)dt_2-
( 4t_1t_2-6t_3)dt_1, \
\eta_2:=( 4t_1t_2-6t_3)dt_3-
(6t_1t_3 -\frac{1}{3}t_2^2)dt_2,
$$
$$
\eta_3:= (t_1^2-\frac{1}{12}t_2)
dt_3-(6t_1t_3 -\frac{1}{3}t_2^2)dt_1,\
\eta_4=3t_3dt_2-2t_2dt_3.
$$
The foliation $\F(\Ra)$ is induced by $\eta_i,i=1,2,3$.
We have
\begin{eqnarray*}
d\Delta(\Ra) & = & (2.27t_3dt_3-3t_2^2dt_2)(\Ra) \\
 & =& 2.27t_3 (6t_1t_3-\frac{1}{3}t_2^2)-3t_2^2(4t_1t_2-6t_3) \\
 &= & 12t_1\Delta.
\end{eqnarray*}
This implies that the variety
$\Delta_0:=\{\Delta=0\}$ is invariant by the foliation $\F(\Ra)$. Inside
$\Delta_0$ we have the algebraic leaf $\{(t_1,0,0)\in\A^3\}$ of $\F(\Ra)$.
We parameterize $\Delta_0$  by $(3t^2,t^3),\ t\in \C$ and conclude that
(\ref{raman}) restricted to $\Delta_0$ is given by
\begin{equation}
\label{raman0}
\F(\Ra)\mid_{\Delta_0}: \left \{ \begin{array}{l}
\dot t=2t_1t-t^2 \\
\dot t_1=t_1^2-\frac{1}{4}t^2
\end{array} \right. .
\end{equation}
It has the first integral $\frac{t_1^2}{t}-t_1+\frac{1}{4}t$. 
This implies that the leaves of $\F(\Ra)$ inside $\Delta_0$ are given by:
$$
t_3^{\frac{1}{3}}-2((t_1+c)^2-t_1^2)^{\frac{1}{2}}=2(t_1+c),\ c\in\C.
$$

\begin{prop}
\label{16apr}
The following is a uniformization of the foliation $\F(\Ra)$
restricted to $T$:
$$
u: \uhp\times (\A^2\backslash \{(0,0)\})\rightarrow T,
$$
\begin{equation}
\label{9apr05}
(z,c_2,c_4)\rightarrow g(z)\bullet
\mat{(c_4z-c_2)^{-1}}{c_4}{0}{c_4z-c_2}=
\end{equation}
$$
(g_1(z)(c_4z-c_2)^2+(c_4z-c_2), g_2(z)(c_4z-c_2)^4, g_3(z)(c_4z-c_2)^6).
$$
\end{prop}
\begin{proof}
One may check directly that for fixed $c_2,c_4$ the map induced by $u$
is tangent to (\ref{raman}) which implies the Proposition.
We give another proof which uses the period map:
From (\ref{rosa}) we have
$$
d(\per)(t)=\frac{1}{\Delta}\per(t)\mat {\frac{3}{4}\eta_2}{\frac{3}{2}\eta_4}
{\frac{9}{2}t_3\eta_1-3t_2\eta_3+\frac{3}{2}t_1\eta_2}
{-\frac{3}{4}\eta_2 }^{\tr}.
$$
Therefore,
$$
d(\per(t))(\Ra(t))=\per(t)\mat{0}{0}{*}{0}=\mat{*}{0}{*}{0}.
$$
This implies that the $x_2$ and $x_4$ coordinates of
the pull forward of the vector field $\Ra$ by $\per$ are zero. Therefore, the
leaves of $\F(\Ra)$ in the period domain are of the form
$$
\mat{z(c_4z-c_2)^{-1}}{c_2}{(c_4z-c_2)^{-1}}{c_4}=
\mat{z}{-1}{1}{0}
\mat{(c_4z-c_2)^{-1}}{c_4}{0}{c_4z-c_2}.
$$
\end{proof}

\section{Proof of Theorem \ref{foli}}
\label{proofof}
We follow the notations introduced in \S \ref{uniformi}. In
particular we work with the family (\ref{khodaya1}) with $t_0=1$. 

{\it Proof of 1.}  The first part follows from Proposition \ref{16apr}. The
leaves of the pull-forward of the foliation $\F(\Ra)$ by the period
map $\per$ have constant $x_2$ and $x_4$ coordinates. By definition
of $B_{\frac{xdx}{y}}:=\Im(x_2\overline{x_4})$ in the period domain, we conclude that
$M_r$'s are $\F(\Ra)$-invariant. 
On $M_0$ an
$x\in \pedo$ can be written in the form $\mat
{x_1}{x_4r}{x_3}{x_4},\ r\in\R, x_4(x_1-rx_3)=1$. Then
\begin{equation}
\label{10apr05}
B_{\frac{dx}{y}, \frac{xdx}{y}}(x)=\overline{x_4}(x_1-rx_3)=\frac{\overline{x_4}}{x_4}.
\end{equation}
which implies that $N_w$'s are $\F(\Ra)$-invariants.

{\it Proof of 2.} Let us define
$$
L_{c_2,c_4}:=
\left \{\mat{z(zc_4-c_2)^{-1}}{c_2}{(zc_4-c_2)^{-1}}{c_4} \mid z\in \uhp\backslash\{\frac{c_2}{c_4}\}
\right
\}.
$$
We look at a leaf $L_{c_2,c_4}$ of $\F(\Ra)$ at the period 
domain $\pedo$. The leaf $[L_{c_2,c_4}]\subset \SL 2\Z\backslash \pedo $ may 
not be biholomorphic to $\uhp\backslash\{\frac{c_2}{c_4}\}$ 
if there exists $A\in\SL 2\Z$ which maps a point of  
$L_{c_2,c_4}$ to another point in $L_{c_2,c_4}$. This implies that  
$A[c_2,c_4]^{\tr}=[c_2,c_4]^{\tr}$ and hence  $\frac{c_4}{c_2}\in\Q$. After taking another representative for
the leaf $[L_{c_2,c_4}]$, we can assume that $c_4=0$. Now, the only elements
of $\SL 2\Z$ which maps $[c_2,0]$ to itself are of the form $\mat{1}{b}{0}{1},b\in \Z$. This implies that
the corresponding leaf in $T$ is biholomorphic to $\D\backslash \{0\}$.
If $B_{\frac{xdx}{y}}(t)\leq 0$ and $t\not\in K$, then $\frac{c_2}{c_4}\not\in\uhp$ and $L_t$ is biholomorphic to 
$\uhp$. If $B_{\frac{xdx}{y}}(t)>0$ then $\frac{c_2}{c_4}\in \uhp$ and $L_t$ is biholomorphic to $\uhp\backslash\{\frac{c_2}{c_4}\}$. 

{\it Proof of 3.}
Take $t\in K$ and a cycle $\delta\in
H_1(E_t,\Z)$ such that $\int_\delta\frac{xdx}{y}=0$ and
$\delta$ is not of the form $n\delta'$ for  some $2\leq n\in \N$ and
$\delta'\in H_1(E_t,\Z)$. We choose another $\delta'\in
H_1(E_t,\Z)$ such that $(\delta',\delta)$ is a basis of
$H_1(E_t,\Z)$ and the intersection matrix in this basis is
$\mat{0}{1}{-1}{0}$. Now $\per(t)$ has zero $x_4$-coordinate and so
its $B_{\frac{xdx}{y}}$ is zero. This implies that $K\subset M_0$. It is dense because
an element $\mat{x_1}{x_4r}{x_3}{x_4}\in M_0\subset \L$ 
can be approximated by the elements in $M_0$ with $r\in \Q$.  

The image of the map $g$ is the locus of the points $t$ in $T$
such that $\per(t)$ is of the form $\mat{z}{-1}{1}{0}$ in a basis
of $H_1(E_t,\Z)$. We look $g$ as a function of $q=e^{2\pi
iz}$ and we have
$$
g(0)=p_\infty,\ \frac{\partial g}{\partial
q}(0)=(-24a_1,240a_2,-504a_3)
$$
where $a_i$'s are defined in (\ref{pinfty}).
This implies that the image of $g$ intersects $\sing(\Ra)$ transversely.
For $t\in K$ the $x_4$-coordinate of $\per$ is zero and the leaf
through $t$, namely $L_t$, has constant $x_2$-coordinate, namely
$c_2$. By (\ref{9apr05}) $L_t$ is uniformized by
$$
u(z)=(c_2^2g_1(z), c_2^4g_2(z),c_2^6g_3(z)),\ z\in\uhp .
$$
This implies that $L_t$ intersects $\sing(\Ra)$ transversely at
$(c_2^2a_1,c_2^4a_2,c_2^6a_3)$.

Note that the leaf space $\SL 2\Z \backslash(\C^2\backslash \{\Im(c_2\bar c_4)>0\})$ 
of the foliation $\F(\Ra)$ in $M_{r>0}$   is biholomorphic to the quasi affine set 
$\C^2 \backslash \{27t_3^2-t_2^3=0\}$ using the Eisenstein series. The same is true 
for $M_{r<0}$. The leaf space in $M_0$ is isomorphic to
$\C^{*}\times \SL 2\Z \backslash \R $ as a set and so has no reasonable structure. 

{\it Proof of 4.} Let $t\in T$ and the leaf $L_t$ through $t$ have an
accumulation point at $t_0\in T$. We use the period map $\per$ and
look $\F(\Ra)$ in the period domain. For $(c_2,c_4)\in
\A^2\backslash \{0\}$ the set $S=\{A(c_2,c_4)^\tr\mid A\in \SL
2\Z\}$ has an accumulation point in $\A^2$ if and only if
$\frac{c_2}{c_4}\in\R\cup{\infty}$ or equivalently $B_{\frac{xdx}{y}}(t)=0$.

{\it Proof of 5.} It is already proved in \S \ref{uniformi}.
\section{Elliptic modular foliations}
\label{emf}
Let $\eta$ be any meromorphic differential 1-form in $\C^2$ whose restriction to
a smooth  elliptic curve $E_t$ gives us a differential form of the second type.
For instance, one can take $\eta=\frac{p(x,y)dx}{y}$ or $p(x,y)(3xdy-2ydx)$, where $p$ is a 
polynomial in $x,y$. Such a $1$-form can be written in the form
\begin{equation}
\label{26feb06}
\eta=p_1(t)\frac{dx}{y}+p_2(t) \frac{xdx}{y}\ \hbox{ modulo relatively exact $1$-forms},
\end{equation}
where $p_1$ and $p_2$ are two meromorphic functions in $t$ with poles along $\Delta=0$ (a meromorphic one form $\eta$ in $\C^2$ is called relatively exact if its
restriction to each smooth elliptic curve $E_t$ is an exact form). 

An elliptic modular foliation $\F_\eta$ associated to $\eta$ is a foliation in 
$\C^3=\{(t_1,t_2,t_3)\}$ given locally by  the constant locus of the integrals 
$\int_{\delta_t}\eta, \delta_t\in H_1(E_t,\Z)$,
i.e. along the leaves of $\F_\eta$ the integral $\int_{\delta_t}\eta$ as a function in 
$t$ is
constant. The algebraic description of $\F_\eta$ is as follows:
We write $ \eta=p\omega$, where $\omega=(\frac{dx}{y},\frac{xdx}{y})^\tr$ and $p=(p_1,p_2)$. 
If  $\nabla\omega=B\omega$ is the Gauss-Manin connection 
of the family (\ref{khodaya}) with respect to the basis $\omega$ (see \S \ref{ramanfoli}) then
$$
\nabla(\eta)=\nabla(p\omega)=(dp+pB)\omega
$$
and it is easy to see that
\begin{equation}
\label{mondai}
\F_\eta: dp_1+p_1\omega_{11}+p_2\omega_{21}=0,\ dp_2+p_1\omega_{12}+p_2\omega_{22}=0,
\end{equation}
where $B=\mat{\omega_{11}}{\omega_{12}}{\omega_{21}}{\omega_{22}}$.
By the first part of Theorem \ref{foli} we know that $\F_{\frac{xdx}{y}}=\F(\Ra)$. 
Using the above expression for $\F_\eta$ one can show that $\{\Delta=0\}$ is $\F_\eta$-invariant
and every leaf of $\F_\eta$ inside $\{\Delta=0\}$ is algebraic. 
\begin{exam}\rm
For $s$ a fixed  complex number, the foliation $\F_{\frac{(s+x)dx}{y}}$ is given by the vector field:
$$
(t_1^2+2t_1s-\frac{1}{12}t_2+s^2)\frac{\partial}{\partial t_1}+
(4t_1t_2+4t_2s-6t_3)\frac{\partial}{\partial t_2}+
(6t_1t_3-\frac{1}{3}t_2^2+6t_3s)\frac{\partial}{\partial t_3}
$$
For $s=0$ this is the foliation $\F(\Ra)$ discussed in the previous sections and
for $s=\infty$ this is the trivial foliation $\F_{\frac{dx}{y}}: dt_2=0,\ dt_3=0$.
\end{exam}
\begin{exam}\rm
We have $\frac{x^2dx}{y}=(-t_1^2+\frac{1}{12}t_2)\frac{dx}{y}+2t_1\frac{xdx}{y}$ modulo
relatively exact forms and so $\F_{\frac{x^2dx}{y}}$ is given by:
$$
(-48t_1^4+24t_1^2t_2-48t_1t_3+t_2^2)\frac{\partial}{\partial t_1}+
  (-384t_1^3t_2+1728t_1^2t_3-96t_1t_2^2+48t_2t_3)\frac{\partial}{\partial t_2}+
  $$
  $$
  (-576t_1^3t_3+96t_1^2t_2^2-144t_1t_2t_3-8t_2^3+288t_3^2)
   \frac{\partial}{\partial t_3}.
$$
\end{exam}

\section{Abelian subvariety theorem}
\label{ast}
In this section we are going to state a consequence of the abelian subvariety theorem
on periods of an abelian variety defined over $\bar \Q$. For the convenience of the reader,
we recall some basic facts about abelian varieties. For further information the reader
is referred to \cite{labi} for  the analytic theory  and \cite{milab} for the arithmetic theory of
abelian varieties.
 
An abelian variety $A$ viewed as a complex manifold is biholomorphic to $\C^g/\Lambda$, where
$\Lambda$ is a lattice of rank $2g$ in $\C^g$. In addition we have an embedding of $A$ in some
projective space which makes sense to say that $A$ is defined over $\bar\Q$. From now on, we
work only with the category of abelian varieties defined over $\bar \Q$. 
According to Grothendieck \cite{gro}  the de Rham cohomology $H_{\dR}^1(A)$ can be constructed
in the context of algebraic geometry and it is a $\bar \Q$-vector 
space of dimension $2\dim(A)$. Every 
$[\omega]\in H^1_\dR(A)$ is represented by a differential form $\omega$ of the first or  second type defined over $\bar\Q$. 
A differential 1-form $\omega$ on $A$ is called to be 
of the first type if it is holomorphic on $A$ and it is called to be of the second type if it is 
meromorphic with poles but no residues around the poles.
Let $A_1,A_2$ be two abelian varieties of the same dimension defined over $\bar \Q$.
 An isogeny between $A_1$ and $A_2$ is a surjective morphism  $f:A_1\rightarrow A_2$ of 
 algebraic varieties defined over $\bar \Q$ with $f(0_{A_1})=0_{A_2}$. 
 It is well-known that every isogeny is a group homomorphism and there is another isogeny $g:A_2\rightarrow A_1$ such that 
$g\circ f=n_{A_1}$ for some $n\in\N$, where $n_{A_1}$ is the multiplication by $n$ map in $A_1$. 
The isogeny $f$ induces an isomorphism $f_*:H_1(A_1,\Q)\rightarrow H_1(A_2,\Q)$( resp. 
$f^*:H_\dR^1(A_2)\rightarrow H_\dR^1(A_1)$) of $\Q$-vector spaces (resp. $\bar\Q$-vector spaces). 
For $A=A_1=A_2$  simple, it turns out that $\End_0(A)=\End(A)\otimes_\Z\Q$ is a division algebra, i.e. it is a ring, possibly
non-commutative, in which every non-zero element has an inverse.
An abelian variety is called simple if it does not contain a non trivial abelian subvariety.
Every abelian variety is isogenous to the
direct product $A_1^{k_1}\times A_2^{k_2}\times\cdots \times A_n^{k_n}$
of simple, pairwise non-isogenous abelian varieties $A_i$, all defined over
$\bar\Q$ and this decomposition is unique up to isogeny and permutation
of the components. For an abelian variety $A$ defined over $\bar \Q$ the period set
$$
P(A):=\{\int_{\delta}\omega\mid \delta\in H_1(A,\bar\Q),\ [\omega]\in H_\dR^1(A)\}
$$
is a $\bar \Q$-vector space of dimension at most $(2\dim A)^2$. We are going to
state the precise description of $\dim_{\bar\Q}P(A)$.

Let $A$ be a simple abelian variety.
The division algebra
$k:=\End_0(A)$ acts both on $H_1(A,\Q)$ and $H_\dR^1(A)$ and we have
$$
\int_{a\cdot \delta}\omega=\int_{\delta}a\cdot \omega,\ a\in k,\ [\omega]\in H_\dR^1(A).
$$ 
This means that the periods of $a\cdot \delta$ reduces to the periods of $\delta$.
Let $H_1(A,\Q)=\oplus_{j=1}^s k\cdot \delta_j$  be the decomposition of $H_1(A,\Q)$ under the action
of $k$. Each $k\cdot \delta_j$ is a $\Q$-vector space of dimension $\dim_\Q k$ and so
$s=\frac{\dim_{\Q}H_1(A,\Q)}{\dim_\Q k}=\frac{2\dim(A)}{\dim_{\Q} (\End_0(A_i))}$.
Considering $r=2\dim A$ differential forms  $\omega_1,\omega_2,\cdots,\omega_r$ 
which form a basis of $H^1_\dR(A)$, we obtain
$s_A:=\frac{4\dim(A)^2}{\dim_{\Q} (\End_0(A_i))}$ periods $\int_{\delta_j}\omega_i,\ i=1,2,\ldots,r,\ j=1,2,\ldots,s$ which span the $\bar \Q$-vector space $P(A)$ and may be $\bar\Q$-independent. 
If $A$ is  isogenous to the direct product $A_1^{k_1}\times A_2^{k_2}\times\cdots \times A_n^{k_n}$
of simple, pairwise non-isogenous abelian varieties $A_i$, all defined over
$\bar\Q$ , then we obtain  $\sum_{i=1}^ns_{A_i}$ periods which span the 
$\bar \Q$-vector space $P(A)$. In fact, they form a basis and there is no more relation
between the periods of $A$:
\begin{theo}
\label{dimension}
Let $A$ be an abelian variety defined over $\bar\Q$ and isogenous to the
direct product $A_1^{k_1}\times A_2^{k_2}\times\cdots \times A_n^{k_n}$
of simple, pairwise non-isogenous abelian varieties $A_i$, all defined over
$\bar\Q$.
Then the $\bar \Q$-vector space $V_A$ generated by $1, 2\pi i$ together 
with all periods $\int_{\delta}\omega,\ \delta\in H_1(A,\Q),\ [\omega]\in H^1_\dR(A)$, has dimension
$$
\dim_{\bar\Q}(V_A)=2+4\sum_{i=1}^n\frac{\dim(A_i)^2}{\dim_{\Q} (\End_0(A_i))}.
$$
\end{theo}
Note that the above theorem says a little bit more: The collection of $s_A$ periods
which we described before are $\bar \Q$-linear  independent among themselves and even with
the numbers $1,\pi$.   
The above theorem is  a consequence of W\"ustholz analytic subgroup theorem 
(see for instance 
\cite{shwo95} Lemma 1). It is stated and proved in Theorem 6.1 of \cite{shtswo} (appendix).
Similar theorems are stated and used by many authors (see \cite{wowu} Satz 1, Satz 2. 
\cite{shwo95} Proposition 2, \cite{hoes} Corollary 1). 
In this text we need the following corollaries of the above theorem. 
\begin{coro} 
\label{stefan}
Let $A_1$ and $A_2$ be two abelian varieties over $\bar\Q$ 
with a common non-zero period, i.e. there exist 
$[\omega_i]\in H_\dR^1(A_i),\ \delta_i\in H_1(A_i,\Q), \ i=1,2$ 
such that 
$\int_{\delta_1}\omega_1=\int_{\delta_2}\omega_2\not =0$. 
Then there is sub abelian
varieties $B_1$ of $A_1$ and $B_2$ of $A_2$ with $B_1$ isogenous to $B_2$. 
In particular, if $A_1$ and $A_2$ are simple then
$A_1$ is isogenous to $A_2$. In this case, we have an isogeny $a:A_1\rightarrow A_2$ 
such that $a^{*} [\omega_2]=n[\omega_1]$ and $a_*\delta_1=n\delta_2$ for some $n\in\N$, where 
$a^{*}:H_\dR^1(A_2)\rightarrow H_\dR^1(A_1)$ and $a_*:H_1(A_1,\Q)\rightarrow H_1(A_2,\Q)$
are the induced maps in the first cohomology, respectively homology. 
\end{coro}
Note that all the abelian varieties and isogenies in the above corollary are defined over $\bar\Q$.
\begin{proof}
If there is no common factor in the decomposition of $A_1$ and $A_2$ into simple
abelian varieties then applying Theorem \ref{dimension} to $A_1$ and $A_2$ and $A_1\times A_2$
we conclude that $\dim_{\bar \Q}P(A_1\times A_2)=\dim_{\bar \Q}P(A_1)+\dim_{\bar \Q}P(A_2)$.
This implies that $P(A_1)\cap P(A_2)=\{0\}$ which contradicts the hypothesis.

Now, let us prove the second part. Choose an isogeny $b:A_1\rightarrow A_2$ and 
let $\tilde\delta_2=b_{*}^{-1}\delta_2$ and $\tilde \omega_2=b^{*} \omega_2$. 
Since $\int_{\tilde \delta_2}{\tilde \omega_2}=\int_{\delta_1}\omega_1\not =0$, 
there must be $c\in \End_0(A_1)$ with $c\cdot \delta_1=\tilde\delta_2$, otherwise
by our hypothesis and Theorem \ref{dimension} applied for $A_1$, 
 we will get less dimension for $P(A_1)$. We choose $n\in\N$ such
that $d:=n\cdot c\in \End(A_1)$  and so we have $d_{*} \delta_1=n\tilde \delta_2$. 
By our hypothesis we have 
$$
\int_{\delta_1}d^{*}\tilde \omega_2=\int_{d_*\delta_1}\tilde \omega_2=
n\int_{\delta_1}\omega_1
$$ and
so by Theorem \ref{dimension} we must have $d^{*}\tilde \omega_2=n\omega_1$
(for this one can also use \cite{wowu},  Satz 2).
Now, $e=b\circ d:A_1\rightarrow A_2$ has the properties: $e_*\delta_1=n\delta_2,\ e^*[\omega_2]=n[\omega_1]$.  
\end{proof}
I do not know whether Corollary \ref{stefan} is true for $n=1$ or not. To obtain $n=1$ we have to make
more hypothesis.
\begin{coro}
\label{banan}
Let $A_i,\ i=1,2$ be two simple abelian varieties defined over 
$\bar\Q$ and $0\not =[\omega_i]\in H_\dR^1(A_i)$
such that the $\Z$-modules $\{\int_{\delta}\omega_i\mid \delta\in H_1(A_i,\Z)\}$ coincide.
Then there is an isomorphism  $a:A_1\rightarrow A_2$ such that $a^*[\omega_2]=[\omega_1]$. 
\end{coro}
\begin{proof}
We fix $\delta_i\in H_1(A_i,\Z),\ i=1,2$ such that $\int_{\delta_1}\omega_1=\int_{\delta_2}\omega_2\not =0$, apply Corollary \ref{stefan} and obtain an isogeny 
$a:A_1\rightarrow A_2$ with  
$a^{*} [\omega_2]=n[\omega_1]$ and 
$a_*\delta_1=n\delta_2$
 for some  $n\in \N$. 
 We claim that $a_*H_1(A_1,\Z)=nH_1(A_2,\Z)$. 
 For an arbitrary $\delta\in H_1(A_,\Z)$ we have
$$
\int_{a_*\delta}\omega_2=\int_{\delta}n\omega_1=n\int_{\delta'}\omega_2, \ \hbox{ for some }\delta'\in H_1(A_2,\Z)
$$
Therefore, we have $\int_{a_*\delta-n\delta'}\omega_2=0$. Since $A_2$ is simple, by Theorem
\ref{dimension} we have $a_*\delta=n\delta'$ and so $a_*H_1(A_1,\Z)\subset nH_1(A_2,\Z)$.
In the same way we prove that $nH_1(A_2,\Z)\subset a_*H_1(A_1,\Z)$.

Let $A_{1,n}:=\{x\in A_1\mid nx=0\}$ be the $n$-torsion points of $A_1$.
There is an isomorphism $b:A_1\rightarrow A_2$ such that
$b\circ n_{A_1}=a$. To construct $b$ we proceed as follows: 
For a moment assume that  $a^{-1}(0_{A_2})=A_{1,n}$.
The quotient $B:=A_1/A_{1,n}$ is a well-defined abelian group
defined over $\bar\Q$ and the isogenies $a$ and $n_{A_1}$ induce isomorphisms
 $\tilde a : B\rightarrow 
A_2$ and $\tilde n: B\rightarrow A_1$ of abelian varieties. The isomorphism $b:= \tilde a\circ \tilde 
n^{-1}$ satisfies $b\circ n_{A_1}=a $.  In fact it is the one which we want: 
we have
$b_*\delta_1=\frac{1}{n}b_*(n\delta_1)=\frac{1}{n}a_*\delta_1=\delta_2$ and
$b^*\omega_2=\frac{1}{n}nb^*\omega_2=\frac{1}{n}a^{*}\omega_2=\omega_1$.

Let us prove  $ a^{-1}(0_{A_2})=A_{1,n}$ . It is enough to prove this equality in the analytic
context. We identify $t_{A_1}{\cong}t_{A_2}\cong \C^g$, where the first isomorphism is given by the derivative of $a$ at $0_{A_1}$, $H_1(A_i,\Z)\cong \Lambda_i
\subset \C^g$ and obtain a $\Z$-linear map ${\sf a}:\Lambda_1\rightarrow \Lambda_2$ which
induces a $\C$-linear isomorphism $\C^g\rightarrow \C^g$ (we identify $A_i$ with 
$\C^g/\Lambda_i,\ i=1,2$ and $a$ with ${\sf a}$). 
We have ${\sf a}(\Lambda_1)=n\Lambda_2$  and 
$A_{1,n}=\frac{\Lambda_1}{n}/\Lambda_1$. Therefore ${\sf a}A_{1,n}=0$ mod $\Lambda_2$.
If ${\sf a}(x)=0$ mod $\Lambda_2$ then  ${\sf a}(nx)=n\delta={\sf a}(\delta')$ for some
$\delta\in \Lambda_2,\ \delta'\in \Lambda_1$. Since ${\sf a}$ is injective we have $nx=\delta'$ and so
$x\in A_{1,n}$.
\end{proof}

\section{Proof of Theorem \ref{21feb06}}
\label{proof2}
For a modular foliation $\F_\eta$ we define:
$$
K_{\eta}=\{t\in T\mid \int_{\delta}\eta=0 \hbox{ for some } \delta\in H_1(E,\Z)\},
$$
$$
P_\eta:\C^3\rightarrow \C^3,\ 
P_\eta(t):=t\bullet \mat{p_2^{-1}}{p_1}{0}{p_2}=(t_1p_2^2+p_1p_2,t_2p_2^4,t_3p_2^6),
$$
and
$$
\tilde \Delta=\det(D_tP_\eta),
$$
where $p_i,\ i=1,2$ are given by (\ref{26feb06}).
Using the commutative diagram
$$
\begin{array}{cccc}
&T\backslash\{p_2=0\} &  \stackrel{\per}{\rightarrow} & \SL 2\Z\backslash\pedo  \\
 &P_\eta\downarrow  & & \downarrow  \tilde P_\eta \\
& T & \stackrel{\per}{\rightarrow} & \SL 2\Z\backslash\pedo  
\end{array},
$$
where $\tilde P_\eta$ is the map given by the action of $\mat{p_2^{-1}}{p_1}{0}{p_2}$ from left on $\SL 2\Z\backslash\pedo$,
one can show that $P_\eta$ maps
every leaf of $\F_{\eta}$ to a leaf of $\F(\Ra)$ (possibly a point) and so   
$$
D_tP_\eta(X(t))=\check \Delta \cdot \Ra(P(t)) \hbox{ for some } \check \Delta \in \C[t],
$$ 
where $X=\sum_{i=1}^3X_i\frac{\partial}{\partial t_i}$ is a polynomial vector field tangent to 
$\F_\eta$ and
$X_i$'s have no common factors. 

\begin{theo}
\label{nasim}
Let $\F_{\eta}$ be an elliptic modular foliation associated to the 
family (\ref{khodaya}) and $\eta$, where $\eta$ is defined over $\bar\Q$.
The following is true:
\begin{enumerate}
\item
For any point 
$a\in \bar\Q^3\cap (T\backslash \{p_1=p_2=0\})$ we have:
$$
\bar\Q^3 \cap L_a\subset P_\eta^{-1}P_\eta(a).
$$
In particular, for $a\in T\backslash(\{\tilde \Delta=0\}\cup \{p_2=0\})$ the intersection 
$\bar\Q^3\cap L_a$ is finite.
\item
$K_\eta\cap \bar \Q^3$ is the $\bar\Q$-rational points of the algebraic set 
 \begin{equation}
 \label{cap}
 \{t\in T \mid 0=
[\eta |_{E_t}]\in H_\dR^1(E_t)\}.
\end{equation}
\end{enumerate}
\end{theo}
\begin{proof}
Since $\eta$ is define over $\bar\Q$, we have $p_1,p_2\in\bar
\Q(t_1,t_2,t_3)$. 
If a leaf $L$ of $\F_{\omega}$ contains two distinct $a_i\in T,\ i=1,2$ points 
with algebraic coordinates then by the definition of a modular foliation, the
period $\Z$-modules $\{\int_\delta\eta\mid \delta\in H_1(E_{a_i},\Z)\}, \ i=1,2$  
coincide. We apply Corollary  \ref{banan} and conclude that
there is an isomorphism $b: E_{a_1}\rightarrow E_{a_2}$ with $b^*[\omega]=[\omega]$.
Let $b^*[p_2(a_2)^{-1}\frac{dx}{y}]=k\frac{dx}{y},\ k\in\C$. We have
$$
\per(a_1)\mat{k}{p_1(a_1)}{0}{p_2(a_1)}=\per(a_2)\mat{p_2(a_2)^{-1}}{p_1(a_2)}{0}{p_2(a_2)}
$$
Taking determinant of the above equality we get $k=p_2(a_1)^{-1}$ and using 
(\ref{gavril}), we conclude that $P_\eta(a_1)=P_\eta(a_2)$. 

If for some $0\not =\delta\in H_1(E_t,\Z)$ we have $\int_{\delta}\eta=0$ then using Theorem
\ref{dimension} we conclude that $0=[\eta|_{E_t}]\in H^1_\dR(E_t)$.
Note that the set (\ref{cap}) is equal to
$\{p_1(t)=p_2(t)=0\}$.
\end{proof} 

Theorem \ref{21feb06} follows from Theorem \ref{nasim} for $\eta=\frac{xdx}{y}$.
The map $P_{\frac{xdx}{y}}$ is identity and the set (\ref{cap}) is empty.

For the second part of Theorem \ref{21feb06} we give another proof.
Recall the notations in \S \ref{uniformi}. 
Suppose that there is a parameter $t\in T\cap \bar \Q^3$ such
that $\int_{\delta}\frac{xdx}{y}=0$, for some $\delta\in H_1(
E_t,\Z)$. We can assume that $\delta$ is not a multiple of another cycle
in $H_1(E_t,\Z)$. 
The corresponding period matrix of $t$ in a basis
$(\delta',\delta)$ of $H_1(E_t,\Z)$ has zero $x_4$-coordinate
and so the numbers
$$
t_i=F_i\mat{x_1}{x_2}{x_3}{0}x_3^{-2i}\es{i}(\frac{x_1}{x_3}),\ i=2,3,\
t_1=F_1\mat{x_1}{x_2}{x_3}{0}=x_3^{-2}\es{1}(\frac{x_1}{x_3})
$$
are in $\bar{\Q}$. 
This implies that for $z=\frac{x_1}{x_3}\in\uhp$ we have
$$
\frac{\es{3}}{\es{1}^3}(z), \frac{\es{2}}{\es{1}^2}(z),\
\frac{\es{3}^2}{\es{2}^3}(z)\in \bar{\Q}.
$$
This is in contradiction with the following:

{\bf Theorem} (Nesterenko 1996, \cite{nes01})
{\it For any $z\in\uhp$, the set
$$
e^{2\pi i z},\ \frac{g_1(z)}{a_1}, \frac{g_2(z)}{a_2},\frac{g_3(z)}{a_3}
$$
contains at least three algebraically independent numbers over $\Q$. }
\section{The family $y^2-4t_0(x-t_1)(x-t_2)(x-t_3)$}
\label{sarakar}
In this section we consider the family
\begin{equation}
\label{fereydun}
E_t: \ y^2-4t_0(x-t_1)(x-t_2)(x-t_3), \ t\in\C^4
\end{equation}
with the discriminant $\Delta=\frac{-16}{27}\left (t_0(t_1-t_2)(t_2-t_3)(t_3-t_1)\right )^2$.
First, let us identify the monodromy group associated to this family. Fix a smooth elliptic curve $E_t$.
In $H_1(E_t,\Z)$ we distinguish three cycles as follows:
In the $x$-plane, we join $t_{i-1}$ to $t_{i+1}$, $i=1,2,3,\ t_4=t_1,\ t_{-1}=t_3$ by a straight line $\tilde \delta_i$ and 
above it in $E_t$, we consider the closed cycle 
 $\delta_i=\delta_{i,1}-\delta_{i,2}$  which is a double covering of 
$\tilde \delta_i$, 
where by definition $\delta_4=\delta_1$. We assume that the triangle formed by $\tilde \delta_1,\tilde \delta_2$ and
$\tilde \delta_3$ in the $x$-plane is oriented anti-clockwise and so we have:
\begin{equation}
\label{dagigbash}
\langle \delta_i,\delta_{i+1}\rangle=1,\ i=1,2,3.
\end{equation}
Since $H_1(E_t,\Z)$ is of rank $2$, we have $n_1\delta_1+n_2\delta_2+n_3\delta_3=0$ in $H_1(E_t,\Z)$ for some $n_1,n_2,n_3\in\Z$ which
are not simultaneously   zero. The equalities (\ref{dagigbash}) imply that $n_1=n_2=n_3$ and so we have
$\delta_1+\delta_2+\delta_3=0$. A topological way to see this is to assume 
that the oriented triangles 
$\delta_{1,j}+\delta_{2,j}+\delta_{3,j},\ j=1,2$ are homotop to zero in $E_t$.
(for a better intuition take the paths $\tilde \delta_i$ such that the triangle formed by them has
almost zero area). 

We choose $\delta=(\delta_1,\delta_2)$ as a basis of $H_1(E_t,\Z)$. 
Let us now calculate the monodromy group in the basis $\delta$. Since for fixed $t_1,t_2,t_3$ the elliptic
curves $E_t$ with $t_0$ varying are biholomorphic to each other, the monodromy around
$t_0=0$ is trivial. For calculating other monodromies we assume that $t_0=1$.
It is not difficult to see that the monodromy around the hyperplane $t_{i-1}=t_{i+1}$ is given by 
$$
\delta_i\mapsto \delta_{i},\ \delta_{i-1}\mapsto \delta_{i-1}-2\delta_i,\ \delta_{i+1}\mapsto 
\delta_{i+1}+2\delta_{i}.
$$
We conclude that the monodromy group $\Gamma$ in the basis 
$(\delta_1,\delta_2)^\tr$
is generated by:
$$
A_1=\mat{1}{0}{2}{1}, A_2=\mat{1}{-2}{0}{1},\ A_3=\mat{-1}{-2}{2}{3},
$$
where $A_i$ is the monodromy around the hyperplane $t_{i-1}=t_{i+1}$. This is
the congruence group $\Gamma(2)=\{A\in \SL 2\Z \mid A\equiv_2 I\}$ which is
isomorphic to the permutation group in three elements.
Now, we consider the period map $\per:T\rightarrow \Gamma\backslash \pedo$, where 
$T:=\C^4\backslash \{\Delta=0\}$.
The calculation of the Gauss-Manin connection of the family 
(\ref{fereydun}) in the basis $\omega=(\frac{dx}{y},\frac{xdx}{y})^{\tr}$ and hence the derivative of
$\per$ can be done using the map which sends the family 
(\ref{fereydun}) to (\ref{khodaya1}). We have
$$
B=\frac{dt_1}{2(t_1-t_2)(t_1-t_3)}\mat{-t_1}{1}{t_2t_3-t_1(t_2+t_3)}{t_1}+
\frac{dt_2}{2(t_2-t_1)(t_2-t_3)}\mat{-t_2}{1}{t_1t_3-t_2(t_1+t_3)}{t_2}+
$$
$$
\frac{dt_3}{2(t_3-t_1)(t_3-t_2)}\mat{-t_3}{1}{t_1t_2-t_3(t_1+t_2)}{t_3},
$$
where $\nabla\omega=B\omega$. 
As before we can prove that the period map is  a global 
biholomorphism. 
 We look at its inverse $F=(F_0,F_1,F_2,F_3)$ 
 which satisfies:
{\tiny
\begin{equation}
\label{inv}
(DF)_x=\frac{1}{\det(x)}
\begin{pmatrix}
-F_0x_4
&F_0x_3
&F_0x_2
&-F_0x_1
\\F_1F_2x_3+F_1F_3x_3-F_1x_4-F_2F_3x_3
&-F_1x_3+x_4
&-F_1F_2x_1-F_1F_3x_1+F_1x_2+F_2F_3x_1
&F_1x_1-x_2
\\F_1F_2x_3-F_1F_3x_3+F_2F_3x_3-F_2x_4
&-F_2x_3+x_4
&-F_1F_2x_1+F_1F_3x_1-F_2F_3x_1+F_2x_2
&F_2x_1-x_2
\\-F_1F_2x_3+F_1F_3x_3+F_2F_3x_3-F_3x_4
&-F_3x_3+x_4
&F_1F_2x_1-F_1F_3x_1-F_2F_3x_1+F_3x_2
&F_3x_1-x_2
\end{pmatrix}.
\end{equation}
}
It is easy to see that $F_0(x)=\det(x)^{-1}$.
In a similar way as in \S\ref{alggr} we define the action of $G_0$ on $\C^4$ by
\begin{equation}
\label{action1}
t\bullet g:=(t_0k_1^{-1}k_2^{-1},
 t_1k_1^{-1}k_2+k_3k_1^{-1},
 t_2k_1^{-1}k_2+k_3k_1^{-1},
 t_3k_1^{-1}k_2+k_3k_1^{-1}), $$ $$ 
 t\in\C^4,
g=\mat {k_1}{k_3}{0}{k_2}\in G_0
\end{equation}
and it turns out that $\per(t\bullet g)=\per(t)\cdot g,\ t\in\C^4,\ g\in G_0$. Taking
$F$ of this equality we conclude that  $F_i,\ i=1,2,3$ satisfies
\begin{equation}
\label{mayer}
F_i(xg)=F_i(x)k_1^{-1}k_2+k_3k_1^{-1},\ i=1,2,3.
\end{equation}
We define $\theta_i,\ i=1,2,3$ to be the restriction of $F_i$ to $x=\mat{z}{-1}{1}{0},\ z\in\uhp$ 
and consider it as
a function in $z$. Now,  the equalities  (\ref{mayer}) imply that
\begin{equation}
\label{14feb2010}
 (cz+d)^{-2}\theta_i(Az)=\theta_{i}(z)+c(cz+d)^{-1}, \
\mat{a}{b}{c}{d}\in \Gamma, \ i=1,2,3.
\end{equation}
The  first column of  (\ref{inv}) implies that $(\theta_1,\theta_2,\theta_3):\uhp\rightarrow \C^3$
satisfies the   differential equation:
\begin{equation}
\label{ramanu}
\left \{ \begin{array}{l}
\dot t_1=t_1(t_2+t_3)-t_2t_3\\ 
\dot  t_2= t_2(t_1+t_3)-t_1t_3 \\
 \dot  t_3= t_3(t_2+t_3)-t_1t_2
\end{array} \right.
\end{equation}
The foliation induced by the above equations in $\C^3$ has the axis $t_1,t_2$ and $t_3$ as 
a singular set.  
It leaves the hyperplanes $t_i=t_j$ invariant and is integrable there. For instance, 
a first integral in $t_1=t_2$ is given by $\frac{t_1-t_2}{t_2^2}$.    
Considering the map from the family 
(\ref{fereydun}) to (\ref{khodaya1}), we conclude that:
$$
g_1=\frac{1}{3}(\theta_1+\theta_2+\theta_3), \ g_2=4\sum_{1\leq i<j\leq 3}(g_1-\theta_i)(g_1-\theta_j),\
g_3=4(g_1-\theta_1)(g_1-\theta_2)(g_1-\theta_3).
$$ 
We can write the Taylor series of $\theta_i$'s in $q=e^{2\pi i z}$. I do not know statements
similar to Proposition \ref{16feb06} for $\theta_i$'s.
\section{Another basis}
\label{21febr}
Let us consider the family (\ref{khodaya1}). 
Sometime it is useful to use the differential forms
\begin{equation}
\label{3mar05} \eta_1:=\frac{-2}{5}(2xdy-3ydx),\hbox{ and }
\eta_2:=\frac{-2}{7}x(2xdy-3ydx).
\end{equation}
They are related to $\omega_1,\omega_2$ by:
\begin{equation}
\label{classical}
\frac{d\eta_1}{df}=\frac{dx}{y},\ 
\frac{d\eta_2}{df}=\frac{xdx}{y}
\end{equation}
\begin{equation}
\label{18.2.06}
\begin{pmatrix}
\eta_1 \\
\eta_2 
\end{pmatrix}=
\begin{pmatrix}
\frac{4}{5}t_1t_2-\frac{6}{5}t_3 &
-\frac{4}{5}t_2 \\
\frac{1}{105t_0}(84t_0t_1^2t_2-36t_0t_1t_3-5t_2^2) &
-\frac{4}{5}t_1t_2-\frac{6}{7}t_3
\end{pmatrix}
\begin{pmatrix}
\omega_1 \\ \omega_2 \end{pmatrix}.
\end{equation}
Note that the above matrix has determinant $\frac{4}{105t_0}\Delta$ and so
$\eta_i,\ i=1,2$ restricted to a smooth elliptic curve $E_t$ form a basis of $H_\dR^1(E_t)$. 
The calculation of the Gauss-Manin connection with respect to the
basis $\eta=(\eta_1,\eta_2)^{\tr}$ leads to:
$$
A_0=\mat {\frac{21}{2}t_0t_1t_2t_3-9t_0t_3^2+\frac{3}{4}t_2^3} {-\frac{21}{2}t_0t_2t_3}
{\frac{21}{2}t_0t_1^2t_2t_3+9t_0t_1t_3^2-\frac{1}{2}t_1t_2^3-\frac{5}{8}t_2^2t_3}
{-\frac{21}{2}t_0t_1t_2t_3-18t_0t_3^2+\frac{5}{4}t_2^3}
$$
$$
A_ 1 =\mat 0 0 {27t_0^2t_3^2-t_0t_2^3} 0
$$
$$
A_ 2 =\mat {-\frac{63}{2}t_0^2t_1t_3-\frac{5}{4}t_0t_2^2} {\frac{63}{2}t_0^2t_3}
{-\frac{63}{2}t_0^2t_1^2t_3+\frac{1}{2}t_0t_1t_2^2+\frac{15}{8}t_0t_2t_3}
{\frac{63}{2}t_0^2t_1t_3-\frac{7}{4}t_0t_2^2}
$$
$$
A_ 3 =\mat {21t_0^2t_1t_2+\frac{45}{2}t_0^2t_3} {-21t_0^2t_2}
{21t_0^2t_1^2t_2-9t_0^2t_1t_3-\frac{5}{4}t_0t_2^2}
{-21t_0^2t_1t_2+\frac{63}{2}t_0^2t_3}
$$
where $\nabla\eta=(\frac{1}{\Delta}\sum_{i=1}^4A_idt_i)\eta$.
We have
$$
\F_{\eta_1}: \ \frac{\partial}{\partial t_1},\
\F_{\eta_2}: \ (-60t_1^2+5t_2)\frac{\partial}{\partial t_1}+
(48t_1t_2-72t_3)\frac{\partial}{\partial t_2}+
(72t_1t_3-4t_2^2)\frac{\partial}{\partial t_3}
$$
\begin{rem}\rm
Both differential forms $\omega_1$ and $\eta_1$ are invariant under the morphism 
$(x,y)\mapsto (x+s,y)$ and this is the reason why $\F_{\frac{dx}{y}}=\F_{\eta_1}$ is given by
$dt_2=0,\ dt_3=0$. This and the first row of the equality (\ref{18.2.06}) implies that
for constant $t_2,t_3$ the integral $\int_{\delta}\frac{xdx}{y}$ is a degree one  polynomial in $t_1$ and
hence $\nabla^2_{\frac{\partial}{\partial t_1}}\frac{xdx}{y}=0$. This equality can be also checked
directly from the Gauss-Manin connection (\ref{rosa}).

\end{rem}
\begin{rem}\rm
Consider the weighted ring $\R[x,y],\ \deg(x)=2,\ \deg(y)=3$. One can extend the definition of
the degree to the differential 1-forms $\omega$ in $\R^2$ by setting $\deg(dx)=2,\deg(dy)=3$.
Any real holomorphic foliation $\F(\omega)$ in $\R^2$ with $\deg(\omega)= 6$ has no limit
cycles. In fact, we can write $\omega=df-a\eta_1,\ a\in\R$ (up to multiplication by a constant and a linear change of coordinates), 
where $f$ is the polynomial
in (\ref{khodaya}) with $t\in\R^4$, and if $\F(\omega)$ has a limit
cycle $\delta$ then $0=\int_{\delta}{df}=a\int_{\delta}\eta_1=(-2a)\int_{\delta}dx\wedge dy$, which
is a contradiction. Considering $\F(\omega),\ \deg(\omega)=7$, we can write $\omega=df-a\eta_1-
b\eta_2,\ a,b\in\R$ and such a foliation can have limit cycles because the integral 
$\int_{\delta_s}\eta_2$ may have zeros, where $\delta_s$ is  a continuous family of real vanishing 
cycles parameterized by the image $s$ of $f$. 
 To count the zeros of $\int_{\delta_s}\eta_2$ we may do as follows: 
 We choose another cycle $\tilde \delta_s$ such that $\delta_s$ and $\tilde \delta_s$ form
 a basis of $H_1(\{f=s\},\Z)$ with $\langle \delta_s,\tilde\delta_s\rangle=1$. 
 The real valued function $B_2(s)=\Im(\int_{\delta_{s}}\eta_2\overline{\int_{\tilde\delta_{s}}\eta_2})$ is analytic in $\C\backslash\{c_1,c_2\}$, where $c_1$ and $c_2$ are critical values of $f$. It is 
continuous and zero  in $c_1,c_2$ .
 The intersection of the real curve  $B_2=0$ with the real line $\R$
is a bound for the number of zeros of $\int_{\delta_s}\eta_2$.
\end{rem}

\def\cprime{$'$} \def\cprime{$'$}

\bibliographystyle{plain}

\end{document}